\documentclass{amsart}

\usepackage{amsmath}
\usepackage{amsfonts}
\usepackage{amstext}
\usepackage{amsbsy}
\usepackage{amsopn}
\usepackage{amsxtra}
\usepackage{upref}
\usepackage{amsthm}
\usepackage{amsmath}
\usepackage{amssymb}

\newtheorem{prop}{Proposition}[section]
\newtheorem{rem}{Remark}[section]
\newtheorem{lema}{Lemma}[section]
\newtheorem{defi}{Definition}[section]
\newtheorem{teo}{Theorem}[section]                                             
\newtheorem{eje}{Example}[section]

\title[Multifractal analysis for the backward continued fraction map]{Multifractal analysis of Lyapunov exponent for the backward continued fraction map}

\begin{thanks}
{The author wishes to express his gratitude to Jan Kiwi, Mario Ponce and Juan Rivera Letelier for many fruitful discussions and comments during the preparation of this work. The author was partially supported by Proyecto Fondecyt 11070050 and by  Research Network on Low Dimensional Dynamics, CONICYT, Chile. }
\end{thanks}

\author{Godofredo Iommi} 
\address{Facultad de Matem\'aticas, Pontificia Universidad Cat\'olica de Chile (PUC), Avenida Vicu\~na Mackenna 4860, Santiago, Chile}
\email{giommi@mat.puc.cl}
\urladdr{http://www.mat.puc.cl/\textasciitilde giommi/}

\begin{document}

\begin{abstract}
In this note we study the multifractal spectrum of Lyapunov exponents for interval maps with infinitely many branches and a parabolic fixed point. It turns out that, in strong contrast with the hyperbolic case, the domain of the spectrum is unbounded and points of non-differentiability might exist. Moreover, the spectrum is not concave. We establish conditions that ensure the existence of inflection points. To the best of our knowledge this is the first time were conditions of this type are given. We also study the thermodynamic formalism for such maps. We prove that the pressure function  is real analytic in a certain interval and then it becomes equal to zero. We also discuss the  existence and uniqueness of equilibrium measures. In order to do so, we introduce a family of countable Markov shifts that can be thought of as a generalisation of the renewal shift. 
\end{abstract}

\maketitle
\section{Introduction}
Dynamical systems that are sufficiently hyperbolic have the property that almost all orbits move away from each other on time. When the system is defined on a compact space this produces a certain degree of mixing. In such a situation, the orbit structure becomes rather complicated. Therefore, it is of interest to quantify the rate at which the orbits become separated.

In this note we consider (piecewise) differentiable  maps
$T:I \to I,$  where $I \subset \mathbb{R}$ is a bounded union of closed intervals. 
The \emph{Lyapunov exponent} of the map $T$ at the point $x \in I$ is defined by
\[ \lambda_T(x) = \lambda(x) = \lim_{n \to \infty} \frac{1}{n} \log |(T^n)'(x)|, \]
whenever the limit exists.  It measures the exponential rate of divergence of infinitesimally close orbits. 

The Birkhoff ergodic theorem implies that if $\mu$ is an ergodic $T-$invariant measure,  such that $\int \log |T'| \ d \mu$ is finite, then $ \lambda_T(x)$ is constant $\mu$-almost everywhere. Nevertheless, it is possible for the Lyapunov exponent to attain a whole interval of values. In this note we address the problem of describing the range of these possible values and computing the size of the level sets determined by the Lyapunov exponent. More
precisely,   for $\alpha \geq 0$ consider,
\begin{equation*}
J(\alpha)= \Big{\{} x \in I :  \lim_{n \to \infty} \frac{1}{n} \log |(T^n)'(x)| = \alpha \Big{\}},
\end{equation*}
note that it is possible for $J(\alpha)=\emptyset$.  Let
\begin{equation*}
J'= \Big{\{} x \in I : \textrm{the limit} \lim_{n \to \infty} \frac{1}{n} \log |(T^n)'(x)| \textrm{ does not exists}  \Big{\}}.
\end{equation*}
Denote by $\Lambda$ the repeller corresponding to $T$ (see Section \ref{sistemas} for a precise definition).
The set $\Lambda$ can be decomposed in the following way (usually called 
\emph{multifractal decomposition}),
\begin{equation*}
\Lambda = J' \cup \left( \cup_{\alpha} J(\alpha) \right).
\end{equation*}
The function that encodes this decomposition is called \emph{multifractal spectrum of the Lyapunov exponents} and it is defined by 
\begin{equation*}
L(\alpha):= \dim_H(J(\alpha)),
\end{equation*}
where $\dim_H$ denotes the Hausdorff dimension of a set.

This problem  was  first studied by Weiss \cite{w1}, in the context of conformal expanding maps and Axiom A surface diffeomorphisms. He was able to relate the 
Lyapunov exponent with the pointwise dimension of a Gibbs measure. Using results of Pesin and Weiss \cite{pw1} on the multifractal spectrum of pointwise dimension he obtained a description of the Lyapunov spectrum. In particular, he showed that it is real analytic and has bounded domain. Work has also been done in the context of non-uniformly hyperbolic dynamics (there exists Lyapunov exponents  equal to zero). Pollicott and Weiss \cite{pw} and Nakaishi \cite{na} studied the case of the Manneville-Pomeau map (see also the work of Takens and Verbitskiy \cite{tv} and the work of Pfister and Sullivan \cite{ps}). This is an interval map with two branches and a parabolic fixed point at zero. In this case, the Lyapunov spectrum has bounded domain but it can have points where it is not analytic. Recently, Gelfert and Rams \cite{gr2} considered a broader class of such systems and described the Lyapunov spectrum.  Our results are based on their work. The Lyapunov spectrum has also been studied for the Gauss map, this is an interval map with countably many branches and  infinite topological entropy. Pollicott and Weiss \cite{pw} began the study of its Lyapunov spectrum and it was recently completed 
by Kesseb\"ohmer and Stratmann \cite{ks}. It was shown that the Lyapunov spectrum is real analytic and that it has unbounded domain.

This note is devoted to the study of the Lyapunov spectrum of non-uniformly hyperbolic dynamical systems with infinite topological entropy. Our main model is the, so called,
\emph{Renyi} map (see Section \ref{sistemas} for a precise definition). This is a map with a parabolic fixed point at zero and with infinite topological entropy.  It has no absolutely continuous invariant probability measure with respect to the Lebesgue measure.
It is closely related to the backward continued fraction (see Section \ref{sistemas} and \cite{gh, g}). It is also related to the geodesic flow on the modular surface \cite{af}. Schmeling and Weiss \cite{sw} proposed the study of multifractal spectrum of Lyapunov exponent for this map. Note that the Renyi map exhibits all the complicated behaviour of maps with a parabolic fixed point and of the Gauss map. Is the combination of these two features that makes the example interesting.

Our results are also valid for certain interval maps with a parabolic fixed point,  with infinite topological entropy and with a \emph{maximal measure} (that is, an invariant probability measure absolutely  continuous with respect to the $\dim_H(\Lambda)$-conformal measure).  These maps can be thought of  as a generalisation of the Manneville Pomeau maps to systems with infinite topological entropy.
Moreover, we recover the description of the Lyapunov spectrum for the Gauss map.

We prove that, if $T$ is a map with countably many full branches and (possibly) a parabolic fixed point then
\[L(\alpha) = \frac{1}{\alpha} \inf_{t \in \mathbb{R}} (P(-t \log |T'|) +t \alpha),\]
where $P(\cdot)$ denotes the topological pressure. For precise definitions and statements see Sections \ref{termo} and \ref{multi}.

In this infinite entropy  setting, the domain of the spectrum is unbounded and  points of non differentiability might exists. There are level sets for which there are no \emph{measures of full dimension} (that is an invariant measure, $\mu$, such that $\mu(J(\alpha))=1$ and $\dim_H \mu = \dim_H J(\alpha)$). We also give conditions that ensure the existence of inflection points (see Section \ref{i}). We stress that the spectrum is not concave. To the best of our knowledge this is the first time were conditions of this type are given.

One of the main difficulties in this setting is to describe the thermodynamic formalism. That is, the behaviour of the pressure function and the existence and uniqueness of equilibrium measures (see Section \ref{termo}). Two are the sources of difficulties, on the one hand we have to deal with the lack of hyperbolicity produced by the parabolic fixed point, and on the other, the natural symbolic model for these systems is the (non-compact) full-shift on a countable alphabet. We completely describe the thermodynamic formalism for maps with countably many full branches and a parabolic fixed point.
First, we prove existence of equilibrium measures with good ergodic properties corresponding to the potentials $-t \log |T'|$ with  $t \in (t^*, \dim_H \Lambda)$ (see Section \ref{termo} for precise statements). We consider sub-systems of $T$ and construct a sequence of invariant measures that converges to the desired equilibrium measure. Here, the interval structure of the system is used. In order to prove that the equilibrium measure is unique we make use of the theory of countable Markov shifts, as developed by Sarig \cite{sa1}.  The \emph{renewal shift} (see \cite{sa2}) has been used to model the Manneville Pomeau map.  We introduce a family of Markov shifts, that can be thought of as a generalisation of the  \emph{renewal shift}, which serve as symbolic models for maps with $N$-full  branches and a parabolic fixed point. The main feature of these Markov shifts is that they model the system minus the parabolic fixed point and its pre-imaeges.
That is, they serve as model for the hyperbolic part of the system. We describe the thermodynamic formalism for such maps.  We also consider a symbolic model for the hyperbolic part of maps with countably many branches and a parabolic fixed point.
Applying results of Buzzi and Sarig \cite{bs} to these Markov shifts we prove uniqueness of equilibrium measures. We also prove that the pressure function $t \to P(-t \log |T'|)$ is real analytic for  $t \in (t^*, \dim_H \Lambda)$ and equal to zero for $t \geq \dim_H\Lambda$.

Note that the above results are relevant also at a symbolic level. Indeed, the thermodynamic formalism is only completely understood for countable Markov shifts with very simple combinatorics. Thermodynamic formalism for Markov shifts close to the full-shift have been completely studied by Mauldin and Urba\'nski \cite{mu} and by Sarig \cite{sa2}. Also, the thermodynamic formalism for renewal shift was studied by Sarig \cite{sa2}. The Markov shifts we consider here are to be added to this list.

The structure of the paper is as follows. In Section \ref{sistemas} we define the dynamical systems that we are going to work with. Section \ref{termo} is devoted to the thermodynamic formalism for such systems. Existence and uniqueness of equilibrium measures is discussed. In Section \ref{multi} our main result concerning multifractal analysis of Lyapunov exponent is stated. Its differences with the classical hyperbolic setting are discussed. In Section \ref{i} we study conditions that ensure the existence of inflection points. Note that all the Lyapunov spectra considered in this paper have inflection points. The proofs are left to the end of the note.

\section{The dynamical systems} \label{sistemas}
In this section we describe the dynamical systems that we are going to consider.
Let $I=\cup_{n=1}^{\infty} I_n \subset [0,1]$ be a countable (infinite) union of closed intervals with disjoint interiors. Let $T:I \to [0,1]$ be a map such that
\begin{enumerate}
\item $T$ is of class $C^{1+ \epsilon}$  (on each sub-interval).
\item There exists $m \in \mathbb{N}$ and $p \in \cup_{n=1}^{\infty} I_n$
such that $|(T^m)'(x)|>1$ for every $x \in I \setminus \lbrace p \rbrace$.
\item $T(p)=p$ and $|T'(p)| \geq 1$  
\item If $int(I_n)$ denotes the interior of $I_n$, then $\overline{T(int(I_n))}=[0,1]$ for every $n \in \mathbb{N}$.
\item There exists $\gamma>1$  and $C>0$ such that for every $n \in \mathbb{N}$ and $x \in I_n$ we have that $C^{-1} n^{\gamma} \leq |T'(x)| \leq C  n^{\gamma}$. 
\end{enumerate}
Denote by $I_{i_0 \dots i_{n-1}}= \cap_{j=0}^{n-1} T^{-j} I_{i_j}$ the cylinder of length $n$. We also consider a distortion assumption  as in the work of Gelfert and Rams \cite{gr1, gr2}.
We assume the map $T$ to have the \emph{tempered distortion} property,
 that is, there
exists a positive sequence $(\rho_n)_n$ decreasing to zero such that for every $n \in \mathbb{N}$ we have
\[ \sup_{(i_0 \dots i_n)} \sup_{x,y \in I_{i_0 \dots i_{n}}} \dfrac{|(T^n)'(x)|}{|(T^n)'(y)|} \leq \exp(n \rho_n). \]
Such maps  will be called \emph{Markov-Renyi maps} or simply \emph{MR-maps}.
The repeller $\Lambda$  of $T$ is defined by
\[ \Lambda:= \bigcap_{n=0}^{\infty} T^{-n} I.      \]
We will consider three examples, each of one exhibiting one of the possible behaviours 
of the multifractal spectrum.

\subsection{The Gauss map}
An irrational number $ x \in (0,1)$  can be written as a continued fraction of the form
\begin{equation*}
x = \textrm{ } \cfrac{1}{a_1 + \cfrac{1}{a_2 + \cfrac{1}{a_3 + \dots}}} = \textrm{ } [a_1 a_2 a_3 \dots],
\end{equation*}
where $a_i \in \mathbb{N}$. For a general account on continued fractions see \cite{hw, k}. The $n-th$ approximant
$p_n(x) / q_n(x)$ of the number $x \in [0,1]$ is defined by
\begin{equation} \label{approx}
\frac{p_n(x)}{q_n(x)} = \cfrac{1}{a_1 + \cfrac{1}{a_2 + \cfrac{1}{\dots + \frac{1}{a_n}}}} 
\end{equation}
The Gauss map  $G :(0,1] \to (0,1]$, is the interval map defined by 
\[G(x)= \frac{1}{x} -\Big[ \frac{1}{x} \Big], \]
This map is closely related to the continued fraction expansion. 
Indeed, for $0 < x <1$ with $x=[a_1 a_2 a_3 \dots ]$ we have that
$a_1=[1/x], a_2=[1/Gx], \dots, a_n=[1/G^{n-1}x]$. In particular, the Gauss
map acts as the shift map on the continued fraction expansion,
\begin{equation*}
 a_n = \Big[1/G^{n-1}x \Big].
\end{equation*}

The Lyapunov exponent of the Gauss map $G$ at the point $x$ ,
whenever the limit exists, satisfies (see \cite{pw})
\begin{equation} \label{aproximacion}
\lambda(x)=- \lim_{n \to \infty} \frac{1}{n} \log \Big| x - \frac{p_n(x)}{q_n(x)} 
\Big|,
\end{equation}
Therefore, the Lyapunov exponent of the Gauss map quantifies the exponential speed of 
approximation of a number by its approximants (see \cite{pw}). 
Note that in this case we  have that  $\dim_H(\Lambda)=1$. There exists an absolutely continuous ergodic $G-$invariant measure,
 $\mu_G$, called the \emph{Gauss measure}, defined by
\[ \mu_G(A)= \frac{1}{\log 2} \int_A \frac{1}{1+x} dx .\]
where $A \subset [0,1]$ is a Borel set.  From the Birkhoff ergodic theorem we obtain that $\mu_G$-almost everywhere (and hence Lebesgue almost everywhere)
\begin{equation} \label{casitodos}
\lambda(x)=  \frac{\pi^2}{6 \log 2}.
\end{equation}
Note that the range of values of the Lyapunov exponent is $[ 2 \log \left(\frac{1 + \sqrt{5}}{2} \right), \infty )$ (see \cite{ks, pw}).

\subsection{The infinite Manneville-Pomeau map}
This is a generalisation of the well known Manneville Pomeau map \cite{mp}. 
Let $I_n= \Big[ \frac{n-1}{n}, \frac{n}{n+1} \Big)$, for every $n \geq 1$.
The map $T: \cup_{n=1}^{\infty} I_n \to [0,1]$ is defined by 
\[T(x)|I_n=n(n+1) x + \dfrac{1-n}{n+1}, \]
for every $n >1$ and $T|I_1$ is such that 
\begin{enumerate}
\item $\overline{T(I_1)}=[0,1]$,
\item $T'(x)>1$ for every  $x \in I_1 \setminus \lbrace 0 \rbrace$, 
\item $T(0)=0$ and $T'(0)=1$,
\item There exits $\beta>0$ and $\delta >0$ such that 
$T|_{[0,\delta)}(x)=x+x^{1+\beta}$.
\end{enumerate}
Note that in this case $\dim_H(\Lambda)=1$.  If $\beta \in [0,1]$ the map
has a probability invariant measure absolutely continuous with respect to the Lebesgue measure. Note that the orbits spend a large amount of time near the parabolic fixed point. The condition on the class of differentiability implies that the amount of time is not long enough to make the invariant measure infinite (see \cite{lu}). Indeed, in this case
the map is of class $C^{1+ \epsilon}$ but not of class $C^2$. 

If $\beta >1$ then the map has a sigma-finite (but infinite) invariant measure absolutely continuous with respect to the Lebesgue measure.

\subsection{The Renyi map}
The map $R:[0,1) \to [0,1)$ is defined by
\[ R(x)=\frac{1}{1-x} -\Big[\frac{1}{1-x} \Big],  \]
where  $[a]$ denotes the integer part of the number $a$.
It was introduced by Renyi in \cite{re} and we will refer to it as the \emph{Renyi map}. The ergodic properties of this map have been studied, among others, by Adler and Flatto \cite{af} and by Renyi himself \cite{re}. This is a map with infinitely many branches and infinite topological entropy. It has a parabolic fixed point at zero.
It is closely related to the backward continued fraction algorithm \cite{gh,g}. 
Indeed, every irrational number $x \in [0,1)$ has unique infinite backward continued fraction expansion of the form
\begin{equation*}
x = \textrm{ } \cfrac{1}{a_1 - \cfrac{1}{a_2 - \cfrac{1}{a_3 - \dots}}} = \textrm{ } [a_1 a_2 a_3 \dots]_B,
\end{equation*}
where the coefficients $\lbrace a_i \rbrace$ are integers such that $a_i >1$. The Renyi map acts as the shift on the backward continued fraction (see \cite{gh, g}). In particular, 
\[\textrm{If } x =[a_1 a_2 a_3 \dots]_B \textrm{ then } R(x)=[a_2 a_3 \dots]_B.\]
This continued fraction has been used, for example, to obtain results on inhomogenous diophantine approximation (see \cite{pi}).
Note that in this case $\dim_H(\Lambda)=1$ and that the $1-$conformal measure is the Lebesgue measure. Renyi \cite{re} showed that there exists an infinite $\sigma-$finite invariant measure, $\mu_R,$  absolutely continuous with respect to the Lebesgue measure. It is defined by
\[ \mu_R(A) = \int_A \dfrac{1}{x} \ dx, \]
where $A \subset [0,1]$ is a Borel set. 
There is no finite  invariant measure absolutely continuous with respect to the Lebesgue measure.

\subsection{A pathological example} \label{pathological}
The following example does not satisfy condition $(5)$ and will be used in Section \ref{termo} to illustrate the type of phenomena that we are ruling out by imposing this condition.

Let $x(n)=2n(\log 2n)^2$ ad take $N>0$ such that $\sum_{n>N} x(n)^{-1} <1$. Let $(I_n)_n$
be a  sequence of disjoint  intervals such that $I_n \subset [0,1]$ and $|I_n|= x(n)^{-1}$, where $| \cdot |$ denotes the length of the interval. For every $n>1$ define $T|I_n$ as a piecewise linear map of slope
$x(n)$. Let $I_1=[0,x(1)]$ and define $T|I_1$ to be such that 
\begin{enumerate}
\item $\overline{T(I_1)}=[0,1]$,
\item $T'(x)>1$ for every  $x \in I_1 \setminus \lbrace 0 \rbrace$, 
\item $T(0)=0$ and $T'(0)=1$,
\item There exits $\alpha>0$ and $\delta >0$ such that 
$T|_{[0,\delta)}(x)=x+x^{1+\alpha}$.
\end{enumerate}

\section{Thermodynamic formalism} \label{termo}
A major tool in the study of multifractal analysis (and in the dimension theory of dynamical systems in general)  is the thermodynamic formalism.
In the present setting, the version for countable Markov shifts developed by Sarig \cite{sa1}  will be used (see also the work by Mauldin and Urba\'nksi \cite{mu}). 

\begin{defi}
Let $T$ be an MR-map, denote by $ \mathcal{M}_T$ the set of $T-$invariant probability measures. Let $\phi: \Lambda \to \mathbb{R}$ be a continuous \emph{potential}. The \emph{topological pressure} of $\phi$ with respect to $T$ is defined, via the variational principle,  by
\begin{equation*}
P_T(\phi)=P(\phi) = \sup \lbrace h(\mu) + \int \phi \ d\mu :  \mu \in \mathcal{M}_T    
 \textrm{ and } - \int \phi \ d\mu < \infty\rbrace,
\end{equation*}
where $ h(\mu)$ denotes the measure theoretic entropy of $T$ with respect to $\mu$.
\end{defi}

A measure $\mu_{\phi} \in  \mathcal{M}_T$ is called an \emph{equilibrium measure} for $\phi$ if it satisfies:
\[ P(\phi) = h(\mu_{\phi}) + \int \phi \ d\mu_{\phi}. \]

The difficulties to describe the thermodynamic formalism are twofold.  First, we have to deal with the lack of hyperbolicity produced by the parabolic fixed point. Secondly, the natural symbolic model for these maps is the (non-compact) full-shift on an infinite alphabet. As one might expect, these two dynamical features reflects on the behaviour of the pressure. Indeed, for $t < \dim_H(\Lambda)$ the behaviour of  the pressure is governed by the sub-systems with large entropy and for $t>\dim_H(\Lambda)$ it is governed by the parabolic fixed point.  If there is no parabolic fixed point (as in the case of the Gauss map) then it is only the sub-systems of positive entropy that have influence on the pressure.

The next theorems describe the thermodynamic formalism for MR-maps. 
We will rule out the trivial case in which $\log |T'|$ is cohomologous to a constant. That is, the case in which there exists a continuous function $\psi: \Lambda \to \mathbb{R}$ and a constant $c \in \mathbb{R}$ such that
$\log |T'| = \psi - \psi \circ T +c $. We will
assume that the parabolic fixed point $p \in I$ is equal to zero.

\begin{teo} \label{termo2}
Let $T$ be a MR-map. If $T(0)=0$ and $|T'(0)|=1$ then there exists $t^* \geq 0$ such that
\begin{enumerate}
\item If $t < t^*$ then $P(-t\log |T'|)= \infty.$
\item If $t \in (t^*, \dim_H(\Lambda))$ then  $P(-t\log |T'|)$ is finite, positive, strictly decreasing, strictly convex and real analytic.  Moreover, there exists a unique equilibrium measure for $-t \log|T'|.$ 
\item If $t > \dim_H(\Lambda)$ then  $P(-t\log |T'|)=0$ and the Dirac delta at zero, $\delta_0,$ is the only equilibrium measure for  $-t \log|T'|$.
\end{enumerate}
The pressure function is differentiable at $t=\dim_H(\Lambda)$ if and only if  $\delta_0$ is the only equilibrium measure for $ -\dim_H(\Lambda) \log |T'|$.
\end{teo}

The Renyi map, $R$, satisfies the hypothesis of the above theorem. In this case we  have that $t^* =1/2$. Since $\dim_H(\Lambda)=1$ and there is no finite absolutely continuous measure with respect to the Lebesgue measure, the pressure function is
differentiable at $t=1$. 
A MR-map $T$ for which the pressure function behaves like in Theorem \ref{termo2} and which is differentiable at $t= \dim_H(\Lambda)$ will be called \emph{Renyi like}.

On the other hand, the infinite Manneville Pomeau map, when $\beta \in (0,1)$,  is such that $\dim_H(\Lambda)=1$ and it has a finite absolutely continuous measure with respect to Lebesgue. Therefore, the pressure function is not differentiable at $t=1$.
A MR-map $T$ for which the pressure function behaves like in Theorem \ref{termo2} and which is not differentiable at $t= \dim_H(\Lambda)$ will be called \emph{infinite Manneville Pomeau like}.


The main feature of the proof  of  Theorem \ref{termo2} (see Section \ref{proof-termo}) is the construction of a symbolic model for the systems without the parabolic fixed point. This is a generalisation of the method used to study Manneville-Pomeau maps. Using results by Buzzi and Sarig \cite{bs} it is possible to prove that the equilibrium measure is unique. Note that the proof of existence is achieved using approximation arguments on the interval.

If the map $T$ does not have a parabolic fixed point then the thermodynamic formalism was described in \cite{pw},
\begin{teo} [Pollicott-Weiss] \label{termo1}
Let $T$ be a MR-map. If there exists $m >0$ such that $|(T^m)'(x)|>1$ for every $x \in \Lambda$ then there exists $t^* \geq 0$ such that
\begin{enumerate}
\item If $t < t^*$ then $P(-t\log |T'|)= \infty.$
\item If $t >t^*$ then  $P(-t\log |T'|)$ is finite,  strictly decreasing, strictly convex and real analytic. Moreover, there exists a unique equilibrium measure for $-t \log|T'|$.
\end{enumerate}
\end{teo}

The Gauss map, $G$, satisfies the above hypothesis. In this case 
$t^* =1/2$.  Mayer \cite{ma} proved that the pressure function $t \to
P(-t\log |G'|)$ has a logarithmic singularity at $1/2$ and that for $t > 1/2$,  
is real analytic. With our methods we recover the description of the pressure function. Results on ergodic optimization \cite{jmu1} allow us to prove that the slope of the asymptote for $t \to + \infty$ is the golden mean: $(1 +\sqrt{5})/2$.
A MR-map $T$ for which the pressure function behaves like in Theorem \ref{termo1} will be called \emph{Gauss like}.

\begin{eje}
Let $T$ be the \emph{pathological example} defined in subsection \ref{pathological}. The pressure function for this map is such that 
\[
P(-t\log|T'| )=
\begin{cases}
\infty& \text{ if } t < \dim_H(\Lambda),\\
\text{0 }& \text{ if } t \geq \dim_H(\Lambda).
\end{cases}
\]
\end{eje}

\section{The Lyapunov spectrum} \label{multi}
The Lyapunov spectrum of a piecewise uniformly expanding map, $T: \cup_{i=1}^n I_i \to I$, was described by Weiss \cite{w1}. It is real analytic and it has bounded domain.
In our setting both properties might fail simultaneously.  We will prove that the parabolic fixed point can force the Lyapunov spectrum to have points of non-differentiability. The fact that the system has infinitely many branches implies that the spectrum always has unbounded domain. Moreover, it always has an inflection point (see Section \ref{i} for a more detailed discussion).

The formula we obtain for the multifractal spectrum has been obtained in other settings \cite{gr2, ks,na, pw}.  We now state our results regarding the multifractal spectrum. 

\begin{teo} \label{multi-main}
Let $T$ be an MR-map. Then the domain of $L(\cdot)$ is an unbounded sub-interval of $[0, \infty)$ and
\begin{equation}
L(\alpha) = \frac{1}{\alpha} \inf_{t \in \mathbb{R}} (P(-t \log |T'|) +t \alpha).
\end{equation}
Also, $\dim_H(J')= \dim_H(\Lambda)$.
\end{teo}

The following number plays an important role in the description of the multifractal spectrum,
\begin{defi}
Let
\begin{equation} \label{estrella}
\alpha^* :=  \lim_{t \to \dim_H(\Lambda)^-} \frac{d}{dt} P(-t \log |T'|).
\end{equation}
Because of the convexity of the pressure such a limit always exists.
\end{defi}

Since the pressure function has three different types of behaviour this yields to three different types of Lyapunov spectra.

\begin{teo} \label{para}
If $T$ is a MR-map with a parabolic fixed point then the Lyapunov spectrum has unbounded domain and it has an inflection point. Moreover,
\begin{enumerate}
\item if the pressure function is differentiable at $t=\dim_H \Lambda$ then the multifractal spectrum is real analytic; it is strictly decreasing and there exists a measure of full dimension for every level set.
\item If there exists a measure of maximal dimension for the repeller then $\alpha^* >0$ and for every  $\alpha \in [0, \alpha^*]$ we have $L(\alpha)= \dim_H (\Lambda)$  .  If    $\alpha >\alpha^*$ then there exists measures of full dimension for every level set $J(\alpha)$ and the multifractal spectrum is real analytic.
\end{enumerate}
\end{teo}

The Renyi map satisfies the assumptions of  Theorem \ref{para}. Its Lyapunov spectrum has a unique maximum at zero, $L(0)=1$. That is, the set of points for which the Lyapunov exponent is equal to zero has full Hausdorff dimension. The domain of the Lyapunov spectrum is the interval $[0, + \infty)$. It is strictly decreasing and it has an inflection point. Moreover, as in the case of the Gauss map, $\lim_{\alpha \to +\infty} L(\alpha) = 1/2$.

For the infinite Manneville Pommeau map  (with $\beta \in [0,1]$) the Lyapunov spectrum has domain equal to $[0, + \infty)$.  Moreover, $\alpha^*>0$  and for every
$\alpha \in [0, \alpha^*]$ the Lyapunov spectrum has full Hausdorff dimension,
$L(\alpha)=1$. For $\alpha > \alpha^*$ the function $L(\alpha)$ is strictly decreasing and it has an inflection point.

Note that is the influence of the parabolic fixed point that forces the level set of zero to have full Hausdorff dimension. Depending on the tangency of the system $T$ at this point (or equivalently, whether the map admits or not a finite invariant measure absolutely continuous with respect to the $\dim_H(\Lambda)-$conformal measure), the level set corresponding to zero is the only level set of full dimension.

\begin{teo}
If $T$ is an MR-map such that $|(T^m)'(x)|>1$ for some $m>0$ and for every $x \in \Lambda$ then the Lyapunov spectrum has unbounded domain $[\alpha_{\min}, + \infty).$ It is increasing in the interval $[\alpha_{\min}, \alpha^*]$ and decreasing in $[\alpha^*, +\infty)$. It has a unique maximum, it has an inflection point and it is real analytic. There exists measures of full dimension for every level set.
\end{teo}

\begin{rem} \label{amin}
Note that  $\alpha_{\min} =\min \lbrace \lambda(x) : x \in \Lambda \rbrace$ .
\end{rem}

The Gauss map satisfies the assumptions of the above Theorem.
It was shown by Pollicott and Weiss \cite{pw} and by Kesseb\"ohmer and Stratmann \cite{ks} that the domain of $L(\cdot)$ is the interval $[2 \log \left( \frac{1+\sqrt{5}}{2} \right), \infty)$. It has a maximum at $\alpha^* =\frac{\pi^2}{6 \log 2}$ and $\lim_{\alpha \to + \infty} L(\alpha) = 1/2$.  

\begin{rem}
Let $T$ be the \emph{pathological example} defined in subsection \ref{pathological}.
It is not possible to read from the topological pressure the range of values of the Lyapunov spectrum. Indeed, when differentiable the pressure has always derivative equal to zero. Nevertheless, using the approximation argument of subsection \ref{lower}
it is possible to obtain lower bounds for the multifractal spectrum.
\end{rem}

\section{Inflection points of the spectrum} \label{i}
There is almost no discussion in the literature, that we are aware of,  regarding inflection points for the Lyapunov spectrum. It was only explicitly mentioned in the work of Kesseb\"ohmer and Stratmann \cite{ks} for the Gauss map.  Also Barreira and Saussol \cite{bs1} gave some examples of non-concave multifractal spectra, although not for the  Lyapunov spectrum. In the study of multifractal spectrum of pointwise dimension of certain non dynamically defined measures, non-concave multifractal spectra has been observed \cite{bar, te}.

In the present section we establish general conditions under which the Lyapunov spectrum has inflection points. Note that in our setting the Lyapunov spectrum always has such points. Indeed, it is a continuous function, it is concave at its maximum, it is strictly decreasing for sufficiently large values of $\alpha$, it is non-negative and it has unbounded domain.

Let us denote by $\mu_{\alpha}$ the equilibrium measure corresponding to the potential $-t_{\alpha} \log |T'|$ such that $\int \log |T'| \ d\mu_{\alpha} = \alpha$. Note that it is possible that such a measure does not exist, but if it does then
\[L(\alpha) =\frac{1}{\alpha} \Big(P(t_{\alpha}) +t_{\alpha} \alpha \Big) = \dfrac{h(\mu_{\alpha})}{\alpha}. \]

The next Theorem relates the maximum and the inflection points of the Lyapunov spectrum with the derivative (with respect to the dynamical parameter $\alpha$)
of the entropy of the equilibrium measures.

\begin{teo} \label{max-infl}
Let $$\alpha > \min \lbrace \alpha' \in dom(L) : \textrm{ there exists a measure of full dimension } \mu_{\alpha'} \textrm{ for } J(\alpha') \rbrace.$$ Then 
\begin{enumerate}
\item the point $\alpha$ is a maximum of the Lyapunov spectrum $L(\alpha)$ if and only if 
$$ \frac{d}{d \alpha} h(\mu_{\alpha}) = L(\alpha),$$
\item the point $\alpha$ is an inflection point of the Lyapunov spectrum $L(\alpha)$ if and only if 
$$\frac{d^2}{d \alpha ^2} h(\mu_{\alpha}) =2 L'(\alpha).$$
\end{enumerate}
\end{teo}

\begin{rem}   For \emph{Renyi like} and \emph{infinite Manneville Pomeau like} maps we have that
\[\alpha^*=\min \lbrace \alpha' \in dom(L) : \textrm{there exists a measure of full dimension } \mu_{\alpha'} \textrm{ for } J(\alpha') \rbrace, \]
where the number $\alpha^*$ was defined in equation \eqref{estrella}.
For \emph{Gauss like} maps we have that
\[ \alpha_{\min} = 
\min \lbrace \alpha' \in dom(L) : \textrm{there exists a measure of full dimension } \mu_{\alpha'} \textrm{ for } J(\alpha') \rbrace, \]
where $\alpha_{\min}$ was defined in Remark \ref{amin}.
\end{rem}

The next Theorem gives a bound for where an inflection point can be situated.

\begin{teo} \label{ubic-infl}
Every inflection point is larger than $\alpha^*$.
\end{teo}

Note that it is possible for $ \alpha^*=0$. This occurs, for example, for the Renyi map.

\section{Proof of Theorems \ref{termo2} and \ref{termo1}, thermodynamic formalism.} \label{proof-termo}

We start by proving Theorem \ref{termo2}. We introduce a countable Markov shift that can be thought of as generalisation of the renewal shift (see \cite{sa2}). Our proof is a combination of techniques coming from the theory of countable Markov shifts with properties of interval maps. Several of our arguments are of an approximation nature.

\begin{lema}
There exists $s \geq 0$ such that for every $t<s$  the pressure function is infinity, that is $$P(-t \log |T'|) =\infty.$$  
\end{lema}
\begin{proof}
Note that pressure function $ t \to P(-t \log |T'|)$ is non-increasing. Since the system has
infinite entropy we have that $P(0) =\infty$.  By the Bowen formula (see \cite{gr1, pe}) we obtain that
for every $t >\dim_H(\Lambda)$ the pressure function is non positive. The result now follows.
\end{proof}
Denote by $t^*=\sup \left\{ s \in \mathbb{R} : \textrm{ if } t<s \textrm{ then } P(-t \log |T'|) =\infty \right\}$. Note that $t^* \leq \dim_H(\Lambda)$.

\begin{lema}
The pressure function is such that 
\begin{equation*}
\lim_{t \to {t^{*}}}  P(-t \log |T'|) =\infty
\end{equation*}
\end{lema}

\begin{proof}
From hypothesis $(5)$ on the definition of MR-maps we obtain that
\[P(-t \log|T'|) \geq \log \sum_{n=1}^{\infty} n^{-t \gamma}. \]
The result now follows.
\end{proof}

Note that the pathological example of subsection \ref{pathological} does not satisfy the above Lemma. 

\subsection{Equilibrium measures}
We discuss now the existence of equilibrium measures for the potentials $-t \log |T'|$ with $t \in (t^*, \dim_H(\Lambda))$.

Denote by $T_n$  the restriction of the map $T$ to the set $\cup_{i=1}^{n} I_i$ and let $\Lambda_n$ be the corresponding repeller. Let $P_n(\cdot)$ be the topological pressure corresponding to the dynamical system $T_n$.


Let $\alpha >  \alpha^*$ be fixed. Denote by $\mu_n$ the equilibrium measure for the dynamical system $T_n$ corresponding to a potential of the form $-t_n \log |T_n '|$, such that
\[ \int \log |T'| \ d\mu_n = \alpha. \]
Such a measure exists for every $n$ sufficiently large. This is due to the fact that for $t> t^*$ we have $P(-t \log|T'|) < \infty$ and (see \cite{sa1})
\[ \lim_{n \to \infty} P_n(-t \log |T'|) = P(-t \log |T'|). \]
We first show that this sequence has a limit point in $\mathcal{M}_T$ (note that the mass could escape from the domain of the system, as it does for the sequence $\lbrace \delta_p : p \textrm{ fixed  point for } R \rbrace$, where $\delta_p$ denotes the atomic measure supported in $\lbrace p \rbrace$). That is, we rule out the possibility of mass escaping the domain of the system. The following proof works for any accumulation point 
of the boundary of the partition. For simplicity we will assume that point to be equal to $1$.

\begin{lema}
Let $\alpha>\alpha^*$ then the  sequence $\{ \mu_n \}$ is tight in $[0,1)$.
\end{lema}

\begin{proof}
We have that
\[ \alpha =\int_0^1 \log |T'| \ d\mu_n = \int_0^{1- \epsilon} \log |T'| \ d\mu_n  + \int_{1-\epsilon}^1 \log |T'| \ d\mu_n. \]
For every $n \in \mathbb{N}$ there exists $\epsilon >0$ such that
\[ M_{\epsilon}:= \inf \lbrace |\log|T'(x)||  : x \in (1-\epsilon,1) \rbrace \geq n. \]
Let us fix $\epsilon>0$. We have that
\[ \alpha =\int_0^{1- \epsilon} \log |T'| \ d\mu_n  + \int_{1-\epsilon}^1 \log |T'|  \ d\mu_n \geq  \int_0^{1- \epsilon} \log |T'| \ d\mu_n + M_{\epsilon} \mu_n([1-\epsilon, 1]) .\]
Thus,
\[ \dfrac{\alpha - \int_0^{1- \epsilon} \log |T'| \ d\mu_n}{M_{\epsilon}} \geq
\mu_n([1-\epsilon, 1]) .\]
Since $\int_0^{1- \epsilon} \log |T'| \ d\mu_n \geq 0$,  we obtain that
\[\mu_n([1-\epsilon, 1]) \leq \dfrac{\alpha}{M_{\epsilon}}. \]
Therefore,  for every $\delta>0$ there exists $\epsilon(\delta) >0$ such that
$\mu_n([0, 1-\epsilon(\delta)]) > 1 - \delta$.
That is, the sequence $\lbrace \mu_n \rbrace$ is tight. Therefore, if  $\mu_{\alpha}$        is a weak$^*$ accumulation point of the sequence $\lbrace \mu_n \rbrace$ then  $\mu_{\alpha}([0,1))=1$.
\end{proof}

\begin{rem}
Note that the assumption $\alpha > \alpha^*$ is equivalent to $t \in (t^*, \dim_H(\Lambda))$.
\end{rem}

We have proved that given $\alpha > \alpha^*$ there exists a sequence $\lbrace t_n       \rbrace \in \mathbb{R}$ such that the equilibrium measures, $\mu_n$, for $-t_n \log|T'|$ restricted to $\Lambda_n$  has an accumulation point  $\mu_{\alpha}$.  For sufficiently large values of $n$ we have that $t_n \in (t^*, \dim_H(\Lambda))$. In particular it has a 
convergent subsequence. Let us denote by $t_{\alpha}$ such an accumulation point.

\begin{lema}
The measure $\mu_{\alpha}$ is an equilibrium measure for $-t_{\alpha} \log|T'|$.
\end{lema}

\begin{proof}
Recall that results  by Sarig \cite{sa1} (see also the work of  Mauldin and Urba\'nski \cite{mu}) imply that
\[ \lim_{n \to \infty} P_n(-t \log |T'|) = P(-t \log |T'|). \]
Note that the pressure is monotonous on $n$, that is, for every continuous potential $\phi:\Lambda \to \mathbb{R}$ we have
\[ P_n(\phi) \leq P_{n+1}(\phi). \]
In particular, we have that
\[ \lim_{n \to \infty} P_n(-t_n \log |T'|) = P(-t_{\alpha} \log |T'|). \]
From our definition of pressure we obtain
\[\lim_{n \to \infty} \big( h(\mu_n) -t_n \alpha \big) =h(\mu_{\alpha}) - t_{\alpha} \alpha =  P(-t_{\alpha} \log |T'|).\]
\end{proof}

\begin{rem}
If there exists $\alpha_1 \neq \alpha_2$ such that the corresponding sequences
$\lbrace t_n^1 \rbrace$ and $\lbrace t_n ^2 \rbrace$ converge to the same value
$t'$, then the pressure function $t \to P(-t \log |T'|)$ is not differentiable at  $t'$.
\end{rem}

\subsection{The symbolic model: $N$-renewal shift.}

In order to prove that the equilibrium measure is unique, we will make use of the theory of countable Markov shifts.  First note that for every $\alpha > \alpha^*$ the Dirac measure supported at the fixed point zero, $\delta_0$, is not an equilibrium measure for $-t_{\alpha} \log |T'|$.
We will remove the fixed point $x=0$ and use a symbolic model to describe the system $T$ restricted to $(0,1)$.

Let $S =\lbrace 0,1 ,2 ,3 ,\dots \rbrace$ be a countable alphabet. Consider the transition matrix $A_1=(a_{ij})_{i,j \in S}$ with $a_{0,0}=a_{0,n}=a_{n,n-1}=1$ for each $n \geq 1$ and with all the other entries equal to zero. The \emph{renewal shift} is the countable Markov shift $(\Sigma_1, \sigma)$ defined by the transition matrix $A_1$, that is, the shift map $\sigma$ acting on the space
\[ \Sigma_1= \lbrace (x_i)_{i \geq 0}: x_i \in S \textrm{ and } a_{x_{i} x_{i+1}}= 1
\textrm{ for each } i \geq 0 \rbrace. \]
The vertex determined by the symbol $0$ is called \emph{renewal vertex}.
Interval maps with two branches, either hyperbolic \cite{pz} or with a parabolic fixed point \cite{sa1}, have been studied with this symbolic model. Also in dimension two, parabolic horseshoes have been studied with this model \cite{bi}. We will consider a generalisation of the the renewal shift that would allow us to model systems with more than two branches.  The idea is to replace the renewal vertex with a full-shift on $N$ symbols.

Let $N$ be a positive integer. Consider the transition matrix $A_N=(a_{ij})_{i,j \in S}$ with $a_{0,0}=a_{0,n}=1, a_{1,0}=a_{1,n}=1, \dots, a_{N,0}= a_{N,n}=1$ for each $n \geq 1$. Also  $a_{N+1,0}=a_{N+1,1}= \cdots= a_{N+1,N}=1$ and
$a_{n+2,n+1}=1$ for each $n \geq N+2$. 
All the other entries are equal to zero. The \emph{N-renewal shift} is the countable Markov shift $(\Sigma_N, \sigma)$ defined by the transition matrix $A_N$, that is, the shift map $\sigma$ acting on the space
\[ \Sigma_N= \lbrace (x_i)_{i \geq 0}: x_i \in S \textrm{ and } a_{x_{i} x_{i+1}}= 1
\textrm{ for each } i \geq 0 \rbrace. \]
These Markov shifts serve as symbolic models for the systems $T_n$, when we remove the parabolic fixed point and its pre-images.
\begin{prop}
The map $T_n: \Lambda_n \setminus \cup_{i=0}^{\infty} T_n^{-i}(0) \to \Lambda_n \setminus \cup_{i=0}^{\infty} T_n^{-i}(0)$ is topologically conjugated to
$(\Sigma_n, \sigma)$. 
\end{prop}
The proof of this result is straightforward and it is just a slight modification of the classical result for interval maps with two branches (see \cite{bi, pz, sa2}).
Denote by $\pi_n: \Sigma_n \to \Lambda_n \setminus \cup_{i=0}^{\infty} T_n^{-i}(0)$ the topological conjugacy.

Let us consider now the countable Markov shift defined by the transition matrix $A_{\infty}=(a_{i,j})_{i,j \in S}$ with entries $a_{2n,k}=1$,   $a_{2n-1,2n-2}=1$ and
$a_{1,2n}=1$, for every $n,k \in \mathbb{N} \cup \lbrace 0 \rbrace$. 
The \emph{$\infty$-renewal shift} is the countable Markov shift $(\Sigma_{\infty}, \sigma)$ defined by the transition matrix $A_{\infty}$, that is, the shift map $\sigma$ acting on the space
\[ \Sigma_{\infty}= \lbrace (x_i)_{i \geq 0}: x_i \in S \textrm{ and } a_{x_{i} x_{i+1}}= 1
\textrm{ for each } i \geq 0 \rbrace. \]
It can be thought of as a generalisation of the renewal shift, where we replace the renewal vertex by a full-shift on a countable alphabet.
\begin{rem}
The system $(T, \Lambda \setminus \cup_{i=0}^{\infty} T^{-i}(0))$ is topologically conjugated to the system $(\Sigma_{\infty}, \sigma)$. We denote by 
 $\pi: \Sigma_{\infty} \to \Lambda \setminus \cup_{i=0}^{\infty} T^{-i}(0)$ 
the topological conjugacy.
\end{rem}

Let $N \in \mathbb{N} \cup \lbrace \infty \rbrace$.
Given a function $\phi : \Sigma_N \to \mathbb{R}$, for each $n \geq 1$ we set
\[V_n(\phi)= \sup \lbrace |\phi(x) -\phi(y)| : x,y \in \Sigma_N,   x_i = y_i \textrm{ for }
0 \leq i \leq n-1 \rbrace.\]
A function $\phi$ is said to have \emph{summable variations} if
$\sum_{n=2}^{\infty} V_n(\phi) < \infty$.
Note that the space $\Sigma_N$ is  not compact. Sarig \cite{sa1} introduced a notion of pressure in this setting. The \emph{Gurevich pressure} is defined by
\begin{equation*}
P_G(\phi) = \lim_{n \to \infty} \frac{1}{n} \log \sum_{\sigma^n x = x} \exp \Big(
\sum_{i=0}^{n-1} \phi(\sigma ^i x) \Big) 1_{C_{i_0}}(x), 
\end{equation*}
where $1_{C_{i_0}}(x)$ denotes the characteristic function of the cylinder $C_{i_0}$. 
The value of the pressure does not depend on $i_0$. Buzzi and Sarig \cite{bs} proved that if
$(\Sigma, \sigma)$ is a topologically mixing countable Markov shift, and
$\phi: \Sigma \to \mathbb{R}$ is a potential such that $\sup \phi < \infty, P_G(\phi) < \infty$ and $\sum_{n \geq 2} V_n(\phi) < \infty$. Then there exists at most one equilibrium measure for $\phi$.

\begin{prop}
If $\alpha>\alpha^*$ then there exists a unique equilibrium measure for 
$-t_{\alpha} \log |T'|$.
\end{prop}

\begin{proof}
Let $\mathcal{M}_{\infty}$ be the space of $\sigma-$invariant probability measures of $(\Sigma_{\infty}, \sigma)$. The function $m \to m \circ \pi^{-1}$ is a bijection between the ergodic elements of the  sets $\mathcal{M}_T \setminus \lbrace \delta_0 \rbrace$ and
$\mathcal{M}_{\infty}$. 

Note that $\mu_{\alpha}  \circ \pi^{-1}$  is an equilibrium measure for
$-t_{\alpha} \log |T'|  \circ \pi^{-1}$. Therefore, by the result of Buzzi and Sarig \cite{bs} there are no other equilibrium measures in $\mathcal{M}_T \setminus \lbrace \delta_0 \rbrace$. But since 
\[ h(\delta_0) -t_{\alpha} \int \log|T'| \ d\delta_0 <
 h(\mu_{\alpha}) -t_{\alpha} \int \log|T'| \ d\mu_{\alpha}, \]
we have that there exists a unique equilibrium measure for  $-t_{\alpha} \log |T'|$.
\end{proof}

Note that for $t > \dim_H(\Sigma)$  the pressure is non-positive. Since
$$h(\delta_0) -t \int \log|T'| \ d\delta_0 = 0,$$ we have that $P(-t \log |T'|)=0$.

\subsection{Real analyticity of the pressure}
In this subsection we prove that the pressure function $P(-t \log|T'|)$ is real analytic on the range  $t \in (t^*, \dim_H(\Lambda))$. In order to do so, we will make use of an inducing scheme and of the results of the previous subsections regarding equilibrium measures. The method of proof is similar to the one developed by Stratmann and
 Urba\`nski \cite{su}.
 
Assume that the parabolic fixed point belongs to the interval  $I_1$ and let $\tilde{I}=\cup_{n=2}^{\infty}I_n$. Consider the inducing scheme $( \tilde{I}, F,\tau)$, where
$\tau:\tilde{I} \to \mathbb{N}$ is the first return time map and $F:\tilde{I} \to \tilde{I}$ is defined by $F(x)=T^{\tau(x)}(x)$. Note that the map $F$ is piecewise monotonic and expanding ($|F'(x)|>1$). Each interval $I_n$ is divided in a countable number of intervals of monotonicity $I_n=\cup_{j=1}^{\infty}I_{n,j}$. On each of those intervals the map $F$ is surjective. 

Let $t \in (t^*, \dim_H(\Lambda))$ and $\mu_t$ be the unique equilibrium measure corresponding to $-t \log |T'|$. Recall that $\int \log |T'| \ d\mu_t >0$. Results of Bruin and
Todd \cite{bt} and Zweimuller \cite{zw} imply that there exists a unique (Gibbs) $F-$invariant measure $\tilde{\mu_t}$, such that $\int \tau \ d\tilde{\mu_t}< \infty$ which
projects onto $\mu_t$.

Let 
\[P(t,q)=P(-t \log|F'| -q \tau).\]
Since the underlying system is a full-shift on a countable alphabet, the pressure $P(t,q)$, when finite, is real analytic on each variable (see \cite{sa1,su}). Moreover, there exists Gibbs measures \cite{sa3}

Inducing schemes can be though of as suspension flows. In view of that the following result is very much expected.

\begin{lema}
If  $t \in (t^*, \dim_H(\Lambda))$ then $P(t,P(-t \log|T'|))=0$.
\end{lema}

\begin{proof}
Note that by the variational principle and by Abramov formula we have 
\begin{eqnarray*}
P(t,P(-t \log|T'|)) \geq h(\tilde{\mu_t}) -t \int \log|F'| \ d\tilde{\mu_t} -P(-t \log|T'|) \int \tau \ d \tilde{\mu_t} =\\
\int \tau \ d \tilde{\mu_t} \left( \frac{h(\tilde{\mu_t})}{\int \tau \ d \tilde{\mu_t}} -t \frac{\int \log|F'| \ d\tilde{\mu_t}}{\int \tau \ d \tilde{\mu_t}}  -P(-t \log|T'|) \right) =\\
\int \tau \ d \tilde{\mu_t} \left( h(\mu_t) -t \int \log |T'| \ d \mu_t -P(-t \log|T'|)  \right).\end{eqnarray*}
But recall that $\mu_t$ is he unique measure such that
\[  P(-t \log|T'|) = h(\mu_t) -t \int \log |T'| \ d \mu_t.\]
Therefore $P(t,P(-t \log|T'|)) \geq 0$. But note that the inequality can not be strict. Indeed, that would imply that there exists a $T-$invariant  measure $\mu$ (which is obtained as the projection of the Gibbs measure corresponding to $-t  \log|F'| -P(-t \log|T'|) \tau$)   for which 
\[  h(\mu) -t \int \log |T'| \ d \mu> P(-t \log|T'|).    \]
This contradiction with the variational principle proves the statement.
 \end{proof}

\begin{lema}
Let $t \in (t^*, \dim_H(\Lambda))$. If $P(t,q)$ is finite in a neighbourhood of $(t. P(-t \log|T'|))$ then $P(-t \log|T|)$is real analytic.
\end{lema}

\begin{proof}
When finite the function $P(t,q)$ is real analytic. Moreover, $$P(t,P(-t \log|T'|))=0.$$ Applying the implicit function theorem we have that $P(-t \log|T'|)$ is real analytic as long as the non-degenracy condition is satisfied. But indeed,
\[\frac{\partial P(t,q)}{\partial q}\Big|_{(t, P(-t\log|T'|))} = -\int \tau \ d\tilde{\mu_t}, \]
is finite. 
\end{proof} 
 
In order to prove that the pressure function $P(-t \log|T'|)$ is real analytic it only remains to be proven that $P(t,q)$ is finite in a neighbourhood of $(t,P(-t \log|T'|))$. This is indeed the case and the proof of this fact is fairly standard.
 
 \begin{lema}
Let $t>t^*$and $q>0$ then $P(t,q)$ is finite.
 \end{lema}
 
 \begin{proof}
 Note that if
 \[\sum_{Fx=x} \exp \left( -t \log |F'(x)| -q \tau(x) \right) < \infty\]
 then $P(t,q) < \infty$ (see, for example, \cite{sa1}). Denote by $x_{n,j}$ the fixed point of
 $F$ restricted to the interval $I_{n,j}$. We have that
 \begin{eqnarray*}
 \sum_{Fx=x} \exp \left( -t \log |F'(x)| -q \tau(x) \right)  =
 \sum_{n=1}^{\infty}  \sum_{j=1}^{\infty} |F'(x_{n,j})|^{-t} \exp(-qj).  
 \end{eqnarray*}
 The distortion assumption and the renewal shift structure of the inducing scheme (see \cite{sa2}) imply that there exists a constant $C_1>0$ such that $C_1 |I_{n,j}|
\geq |F '(x_{n,j})|$. It also imply that there exists $\rho>1$ and $C_2>0$ such that $ C_2|I_{n,j}| \geq |I_n| j^{-\rho}. $ Therefore, since $t>t^*$ and $q>0$ we have
\begin{eqnarray*}
\sum_{n=1}^{\infty}  \sum_{j=1}^{\infty} |F'(x_{n,j})|^{-t} \exp(-qj) \leq 
C \sum_{n=1}^{\infty}  \sum_{j=1}^{\infty} |I_{n,j})|^{-t} \exp(-qj) \leq \\
C \sum_{n=1}^{\infty} |I_n|  \sum_{j=1}^{\infty} j^{-t \rho} \exp(-qj) \leq
C \sum_{n=1}^{\infty} |I_n| < \infty
\end{eqnarray*}
Hence, $P(t,q)$ is finite.
 \end{proof}
   
Therefore, we have proved that the function $t \to P(-t \log |T'|)$ is real analytic on
$(t^*, \dim_H \Lambda)$.


The proof of Theorem \ref{termo1} is similar to the hyperbolic setting (see \cite{ru}) and can be found in the work of Pollicott and Weiss \cite{pw}. But it can also be obtained
by the same methods applied here, the only difference is that the arguments work 
for $t > t^* $ instead of $t \in (t^*,  \dim_H \Lambda)$.

\section{Proof Theorem \ref{multi-main}, multifractal spectrum.}
In this section we prove the formula for the multifractal spectrum. 

\subsection{The lower bound for the multifractal spectrum.} \label{lower}
\begin{proof}[Proof Theorem \ref{multi-main}]
 In order to obtain the lower bound we use an approximation argument. Let us consider the map $T_n$  which is  the restriction of the map $T$ to the set $\cup_{i=1}^{n} I_i$. The results of Gelfert and Rams \cite{gr2} on multifractal analysis for maps with parabolic fixed points and finite entropy, can be applied to the map $T_n$. We obtain 
\[ L_n(\alpha) = \frac{1}{\alpha} \inf_{t \in \mathbb{R}} (P_n(-t \log |T'|) +t \alpha), \]
where $L_n(\cdot)$ denotes the multifractal spectrum of Lyapunov exponents for the map $T_n$.  Let $dom(L_n)$ denotes the domain of the multifractal spectrum $L_n$, that is, the set of points $\alpha \in \mathbb{R}$ such that the level set, $J_n(\alpha)$, determined by the Lyapunov exponent of $T_n$ is non-empty.
We have that $ dom(L_n) \subset dom(L_{n+1}) \subset dom(L)$, for every $n \in \mathbb{N}$. Note that the family of functions $\lbrace L_n \rbrace$ is monotonous and bounded above by $L$, that is, for every $\alpha \in dom(L_n)$ we have
\begin{enumerate}
\item $L_n(\alpha)  \leq L(\alpha)$
\item $L_n(\alpha) \leq L_{n+1}(\alpha)$
\end{enumerate}
In order to obtain the lower bound, it is sufficient to prove that
\[ \lim_{n \to \infty} L_n(\alpha) = L(\alpha). \]
Note that the sequence $\lbrace   L_n(\alpha) \rbrace $ has a pointwise limit.
The fact that this limit coincides with $L(\alpha)$ follows from the theory of convergence of Fenchel pairs developed by Wijsman \cite{wi1,wi2}. It was  applied to the theory of multifractal analysis in \cite{io} (these ideas were also recently used  in \cite{gr2}).  
Denote by $F$ the Fenchel transform of $P(-t \log |T)$ and by
\[F_n(\alpha):= \sup \lbrace \alpha t - P_n(-t \log |T'|) : t \in \mathbb{R} \rbrace. \]
It was proved by Wijsman that $\lbrace P_n(-t \log |T'|) \rbrace$ converges infimally to $P(-t \log |T|)$ if and only if $\lbrace F_n(\alpha) \rbrace$ converges infimally to
$F$. The notion of infimal convergence is given in equation \eqref{infimal} (see also \cite{wi1, wi2}). In general this notion of convergence does not coincide with the pointwise convergence.  The properties of the pressure, continuity and monotonicity, imply that in this case both notions coincide.

It was shown   by Sarig \cite{sa1} (see also the work of  Mauldin and Urba\'nski \cite{mu}) that the pressure can be approximated in the following way:
\[ \lim_{n \to \infty} P_n(-t \log |T'|) = P(-t \log |T'|). \]
From the properties of the pressure function we obtain that
\begin{equation} \label{infimal}
\lim_{\rho \to 0} \liminf_{n \to \infty} \Big( \inf \lbrace P_n(-t  \log |T'|) : |t -t_0| < \rho \rbrace \Big) =
P(-t_0 \log |T'|).
\end{equation}
That is, the sequence $\lbrace P_n(-t \log |T'|) \rbrace$ converges \emph{infimaly}
to $P(-t \log |T'|)$. Therefore
\[ \lim_{n \to \infty} L_n(\alpha) = L(\alpha). \]
Thus, we obtain the lower bound.
\end{proof}

\subsection{The upper bound for the multifractal spectrum.}
\begin{proof}[Proof Theorem \ref{multi-main}]

If the map $T$ is \emph{infinite Manneville Pomeau like} and
$\alpha \in dom(L)$  with $\alpha < \alpha^*$, then
$$\lim_{n \to \infty} L_n(\alpha) = \dim_H(\Lambda).$$ In particular, for these level sets, the lower bound is also an upper bound. 

Consider now $\alpha \in dom(L)$ such that
there exists a unique invariant  measure $\mu_{\alpha}$ (as it is was shown in Section \ref{proof-termo}) such that $\mu_{\alpha}(J(\alpha))=1$. From the variational principle and the description of the pressure function (Section \ref{termo}) we obtain
\[\dfrac{1}{\alpha}  \inf_{t \in \mathbb{R}} (P(-t \log |T'|) +t \alpha) = \dfrac{P(-t_{\alpha} \log |T'|) +t_{\alpha} \alpha}{\alpha} =\dfrac{h(\mu_{\alpha})}{\alpha}. \]
Moreover, results of Pollicott and Weiss \cite{pw}, imply that
\begin{equation*}
 \underline{\lim}_{r \to 0} \frac{ \log \mu_{\alpha}(B(x,r))}{\log r} \leq \frac{h(\mu_{\alpha})}{\alpha},
\end{equation*}
where $B(x,r)$ the open ball of radius $r$ centered at the point $x$. The result now follows from classical results in dimension theory (see \cite[Theorem $7.2$]{pe}).

\end{proof}

\subsection{The irregular set.}
In an hyperbolic setting it was shown by Barreira and Schmeling \cite{b.s} that
the irregular set, $J'$, has full Hausdorff dimension. Gelfert and Rams \cite{gr2} applied those results to prove that the same holds form the maps $T_n$.
Since
\[ \sup \lbrace \dim_H \Lambda_n : n \in \mathbb{N} \rbrace = \dim_H \Lambda, \]
we have that 
\[\dim_H \Lambda = \dim_H J'. \]

\section{Proof of Theorems \ref{max-infl} and \ref{ubic-infl}, inflection points.}
Let us start by computing the derivative of the multifractal spectrum,
\begin{equation} \label{derivada1}
\dfrac{d}{d \alpha} L(\alpha) = L^{'}(\alpha) =\frac{1}{\alpha ^2} \Big( \alpha \frac{d}{d \alpha} h(\mu_{\alpha}) -  h(\mu_{\alpha}) \Big). 
\end{equation}
Since $P(t_{\alpha}) = h(\mu_{\alpha}) -t_{\alpha} \alpha$ we have
\begin{equation}
\dfrac{d}{d \alpha} P(t_{\alpha}) \cdot \frac{d}{d \alpha} t_{\alpha} = 
\Big( \frac{d}{d \alpha} h(\mu_{\alpha}) \Big) - t_{\alpha} - \alpha \frac{d}{d \alpha} t_{\alpha}.
\end{equation}
Using the formula for the derivative of the pressure (see \cite{ru}) and the fact that the Lyapunov exponent with respect to $\mu_{\alpha}$ is equal to $\alpha$. We obtain that
\begin{equation} \label{derivada-t}
\frac{d}{d \alpha} h(\mu_{\alpha}) = t_{\alpha}.
\end{equation}

\begin{prop}
If for every $\alpha$ in the interior of the domain of the multifractal spectrum there exists a measure of full dimension, $\mu_{\alpha}$, then $L(\alpha)$ has a unique maximum.
\end{prop}

\begin{proof}
Note that from equation \eqref{derivada1} we have that $L'(\alpha)=0$ if and only if
\[ \frac{d}{d \alpha} h(\mu_{\alpha}) = \frac{h(\mu_{\alpha})}{\alpha}. \]
Form equation \eqref{derivada-t}, we obtain that $L(\alpha)$ has a maximum if and only if
\begin{equation} \label{t}
t_{\alpha}= \frac{h(\mu_{\alpha})}{\alpha}.
\end{equation}
Note that
\[\dfrac{P(-t_{\alpha} \log |T'|)}{\alpha} = \dfrac{h(\mu_{\alpha})}{\alpha} -t_{\alpha}. \]
By the Bowen equation the left hand side is equal to zero only for $t_{\alpha} = \dim_H(J)$ and the right hand side only at the local maxima.
\end{proof}

The second derivative of $L(\alpha)$ with respect to $\alpha$ is given by
\begin{equation} \label{derivada2}
\frac{d^2}{d \alpha^2} L(\alpha) =L''(\alpha) = \dfrac{1}{\alpha^3} \Big( \alpha ^2 \frac{d^2}{d \alpha ^2} h(\mu_{\alpha}) - 2 \alpha \frac{d}{d \alpha} h(\mu_{\alpha}) +2 \alpha h(\mu_{\alpha}) \Big).
\end{equation}

\begin{rem} \label{inflection}
The number $\alpha$ is an inflection point for $L(\alpha)$ (that is $L''(\alpha)=0$) if and only if
\begin{equation} \label{condition-inflection}
\dfrac{\alpha ^2}{2} \dfrac{d^2}{d \alpha ^2} h(\mu_{\alpha}) =
\Big( \alpha \frac{d}{d \alpha}  h(\mu_{\alpha})   \Big) - h(\mu_{\alpha}).
\end{equation}
Equivalently,
\begin{equation} \label{condition-inflection1}
\dfrac{d^2}{d \alpha ^2} h(\mu_{\alpha}) =\dfrac{2}{\alpha} \frac{d}{d \alpha}  h(\mu_{\alpha}) -\dfrac{2}{\alpha ^2} =
2 \Big( \frac{1}{\alpha} \Big(  \frac{d}{d \alpha}  h(\mu_{\alpha}) -\frac{h(\mu_{\alpha})}{\alpha} \Big) \Big) = 2 L'(\alpha).
\end{equation}
\end{rem}
We have proved Theorem \ref{max-infl}. In order to prove Theorem \ref{ubic-infl} we need the following lemma

\begin{lema}  \label{con}
The point $\alpha$ is an inflection point of the Lyapunov spectrum $L(\alpha)$ if and only if 
\begin{equation}
P(-t_{\alpha} \log |T'|) = -\frac{\alpha ^2}{2} t_{\alpha}'.
\end{equation}
\end{lema}

\begin{proof}
From equation \eqref{derivada-t} we have that
\[\frac{d}{d \alpha} h(\mu_{\alpha})  = t_{\alpha}   \textrm{ and } \frac{d^2}{d \alpha ^2} h(\mu_{\alpha})  = t_{\alpha} ' .\]
Replacing the above equations in equation \eqref{condition-inflection1} we obtain
\[ t_{\alpha}' =\frac{2}{\alpha} t_{\alpha} -\frac{2}{\alpha} \frac{h(\mu_{\alpha})}{\alpha}. \]
Multiplying by $\alpha / 2$,
\begin{equation} \label{entr-inf}
\frac{h(\mu_{\alpha})}{\alpha} = t_{\alpha} -\frac{\alpha}{2} t_{\alpha}'.
\end{equation}
Since the measure $\mu_{\alpha}$ is an equilibrium measure with Lyapunov exponent equal to $\alpha$, we have that
\begin{equation} \label{pres-inf}
\dfrac{P(-t_{\alpha} \log |t'|)}{\alpha} = \frac{h(\mu_{\alpha})}{\alpha} -t_{\alpha}.
\end{equation}
Combining equations \eqref{entr-inf} and \eqref{pres-inf} we obtain that $\alpha$ is an inflection point if and only if
\begin{equation} \label{inf-computable}
P(-t_{\alpha} \log |T'|) = -\dfrac{\alpha ^2}{2} t_{\alpha}'.
\end{equation}
\end{proof}

\begin{proof} [Proof of Theorem \ref{ubic-infl}]
Consider the function $\alpha \to P(-t_{\alpha} \log |T'|)$ with the same domain as the Lyapunov spectrum $(\alpha_1 , \infty)$ (note that it is possible for $\alpha_1=0$). This pressure function is continuous  and increasing. It is negative for $\alpha \in ( \alpha_1 , \alpha^*)$ and positive for
 $\alpha \in ( \alpha^*, \infty)$.  We have three different cases depending on the behaviour of the pressure function $ t \to P(-t \log |T'|)$.
\begin{enumerate}
\item If the map is \emph{Gauss like} then 
$$\lim_{\alpha \to \alpha_1} P(-t_{\alpha} \log |T'|)  = - \infty \textrm{ and  } \lim_{\alpha \to \alpha_2} P(-t_{\alpha} \log |T'|)  = + \infty$$ 
(in this case $\alpha_1 >0$).
\item  If the map is \emph{Renyi like} then the domain is $[0, + \infty)$. Also,
$$P(-t_{0} \log |T'|)=0 \textrm{ and } P(-t_{\alpha} \log |T'|) >0$$ for every $\alpha>0$.
Moreover,  $\lim_{\alpha \to \infty} P(-t_{\alpha} \log |T'|) = +\infty$.
\item  If the map is \emph{infinite Manneville Pomeau like} then the domain is $[0, + \infty)$. We have that $ P(-t_{\alpha} \log |T'|) =0$ in an interval of the form
$[0, \alpha^*]$ and it is positive and strictly increasing for $\alpha> \alpha^*$.
Note that $$\alpha^* = \lim_{t \to \dim_H(\Lambda)^-} P'(-t \log|T'|).$$  Moreover, $\lim_{\alpha \to \infty} P(-t_{\alpha} \log |T'|) = +\infty$.
\end{enumerate}
Let us study the behaviour of the function $\alpha \to t_{\alpha}$. Again we have three different cases. Denote by $t^*$ the critical value of the pressure (as in Theorems \ref{termo1} and \ref{termo2}).
\begin{enumerate}
\item If the map is \emph{Gauss like} then $\lim_ {\alpha \to  \alpha_1} t_{\alpha} = + \infty$ and $\lim_ {\alpha \to  \infty} t_{\alpha}= p^*$. Moreover, the function
is strictly decreasing. Therefore the function $\alpha \to t_{\alpha}'$ is negative and such that $\lim_ {\alpha \to  \alpha_1} t_{\alpha}' = - \infty$ and
$\lim_ {\alpha \to  \infty} t_{\alpha}'= 0$. Since $\alpha \to P(-t_{\alpha} \log |T'|)$
is negative for $\alpha < \alpha^*$, in virtue of Lemma  \ref{con} there are no inflection points in $(\alpha_1, \alpha^* )$.

\item  If the map is \emph{Renyi like} then the result is clear since $\alpha^*=0$. 

\item  If the map is \emph{infinite Manneville Pomeau like} then the Lyapunov spectrum is constant on the interval $[0, \alpha^*]$.  The function $t_\alpha$ is constant on the interval
 $[0, \alpha^*]$.  Therefore there are no inflection points  in $(0, \alpha^*)$.
\end{enumerate}



\end{proof}


\end{document}